\documentclass[10pt]{article}

\usepackage[margin=0.85in]{geometry}
\usepackage{amsmath,amssymb,amsthm,mathtools}
\usepackage{enumitem}
\usepackage[colorlinks=true,linkcolor=blue,citecolor=blue,urlcolor=blue]{hyperref}
\usepackage{lmodern}

\title{Paving Matroids That Are Not Sparse Paving}

\author{
Mohsen Aliabadi\\
Department of Mathematics, Clayton State University\\
Morrow, GA, USA\\
\texttt{mohsenaliabadi@clayton.edu}
}

\date{}

\numberwithin{equation}{section}

\newtheorem{theorem}{Theorem}[section]
\newtheorem{lemma}[theorem]{Lemma}
\newtheorem{proposition}[theorem]{Proposition}
\newtheorem{corollary}[theorem]{Corollary}
\newtheorem{remark}[theorem]{Remark}
\newtheorem{problem}[theorem]{Problem}

\newcommand{\cB}{\mathcal B}
\newcommand{\cC}{\mathcal C}
\newcommand{\cH}{\mathcal H}
\newcommand{\cS}{\mathcal S}
\newcommand{\Z}{\mathbb Z}

\begin{document}

\maketitle

\begin{abstract}
The Mayhew--Newman--Welsh--Whittle conjecture predicts that asymptotically almost all
matroids are sparse paving. We study the gap between paving and sparse paving matroids at
the logarithmic scale. Let \(p_n\) be the number of paving matroids on \([n]\), let
\(sp_n\) be the number of sparse paving matroids on \([n]\), and let \(sp_{n,r}\) be the
number of rank-\(r\) sparse paving matroids on \([n]\). We prove that
\[
        p_n-sp_n\ge sp_{n,\lfloor n/2\rfloor}^{1-o(1)}.
\]
Thus the paving matroids that are not sparse paving are themselves logarithmically large.
The construction prescribes one hyperplane larger than the rank and then counts stable
sets in an induced subgraph of a Johnson graph. We also give amplified versions obtained
by varying the large hyperplane and by prescribing distance-six families of large
hyperplanes.
\end{abstract}



\section{Introduction}

A matroid is a combinatorial abstraction of linear independence. In this paper all
matroids are finite. If \(E\) is a finite set, a matroid \(M\) on \(E\) may be specified by a
nonempty collection \(\cB(M)\) of subsets of \(E\), called bases, satisfying the basis
exchange axiom: for all \(B,B'\in\cB(M)\) and all \(e\in B\setminus B'\), there exists
\(f\in B'\setminus B\) such that
\[
        B-e+f\in \cB(M).
\]
All bases have the same cardinality; this common cardinality is the rank of \(M\), denoted
\(r(M)\). We use standard matroid terminology, including circuit, flat, hyperplane, and
dual matroid, as in Oxley \cite{Oxl11}.

A rank-\(r\) matroid is called \emph{paving} if every circuit has size at least \(r\). It is
called \emph{sparse paving} if both \(M\) and its dual \(M^*\) are paving. Equivalently, a
rank-\(r\) matroid is sparse paving if every rank-\(r\) nonbasis is a circuit-hyperplane.
In particular, in a rank-\(r\) sparse paving matroid, every hyperplane of cardinality at
least \(r\) has cardinality exactly \(r\). Indeed, if a hyperplane \(H\) has \(|H|\ge r\),
then any \(r\)-subset of \(H\) is a rank-\(r\) nonbasis and hence is a circuit-hyperplane;
maximality of hyperplanes then forces \(|H|=r\). Therefore a rank-\(r\) paving matroid with
a hyperplane of size greater than \(r\) is not sparse paving.

Sparse paving matroids are central in asymptotic matroid enumeration. Let \(m_n\) denote
the number of matroids on the fixed ground set \([n]=\{1,\ldots,n\}\), let \(p_n\) denote
the number of paving matroids on \([n]\), and let \(sp_n\) denote the number of sparse
paving matroids on \([n]\). More generally, let \(sp_{n,r}\) denote the number of rank-\(r\)
sparse paving matroids on \([n]\).

A conjecture of Mayhew, Newman, Welsh, and Whittle \cite{MNWW11}, motivated by earlier
predictions of Crapo and Rota \cite{CR70}, asserts that
\[
        \lim_{n\to\infty}\frac{sp_n}{m_n}=1.
\]
Pendavingh and van der Pol proved the logarithmic version
\[
        \lim_{n\to\infty}\frac{\log sp_n}{\log m_n}=1;
\]
see \cite{PvdP15}. Since \(sp_n\le p_n\le m_n\), it follows that
\[
        \lim_{n\to\infty}\frac{\log p_n}{\log sp_n}=1.
\]
Thus the known logarithmic-scale theory does not separate paving matroids from sparse
paving matroids.

The purpose of this note is to study the complement of the sparse paving class inside the
paving class. The main construction prescribes one hyperplane larger than the rank and
then adds rank-sized hyperplanes indexed by stable sets in an induced subgraph of a
Johnson graph. A lemma of Pendavingh and van der Pol on stable sets in induced subgraphs
of vertex-transitive graphs allows us to count all stable sets in this induced graph,
rather than only those coming from one explicit Graham--Sloane color class.

Our first main result is the following.

\begin{theorem}\label{thm:main}
Let \(p_n\) be the number of paving matroids on \([n]\), let \(sp_n\) be the number of
sparse paving matroids on \([n]\), and let \(sp_{n,r}\) be the number of rank-\(r\) sparse
paving matroids on \([n]\). Then
\[
        p_n-sp_n
        \ge
        sp_{n,\lfloor n/2\rfloor}^{1-o(1)}.
\]
Equivalently,
\[
        \log(p_n-sp_n)
        \ge
        (1-o(1))\log sp_{n,\lfloor n/2\rfloor}.
\]
\end{theorem}

Theorem \ref{thm:main} is a logarithmic-scale result. It does not imply that sparse paving
matroids fail to dominate all matroids, nor even that they fail to dominate paving
matroids. Instead, it gives a quantitative measure of how large the non-sparse part of the
paving class is at the same logarithmic scale as the central middle-rank sparse paving
class.

The proof also gives a fixed-rank amplification. When the prescribed large hyperplane has
size \(r+1\), it is uniquely recoverable from the resulting matroid. Therefore the
fixed-hyperplane count can be multiplied by the number of choices of the large hyperplane.

\begin{theorem}\label{thm:amplified-one}
Let \(r=\lfloor n/2\rfloor\), and set
\[
        \delta_n=
        \frac{(r+1)+\binom{r+1}{2}(n-r-1)}{\binom nr}.
\]
Then
\[
        p_n-sp_n
        \ge
        \binom n{r+1} sp_{n,r}^{1-\delta_n}.
\]
In particular,
\[
        p_n-sp_n
        \ge
        \binom n{\lfloor n/2\rfloor+1}
        sp_{n,\lfloor n/2\rfloor}^{1-o(1)}.
\]
\end{theorem}

We also give a multi-hyperplane construction. The large hyperplanes must form a
distance-six constant-weight code. This gives a clean unconditional lower bound and a
hybrid construction combining large hyperplanes with the induced-Johnson stable-set
method; see Theorems \ref{thm:distance-six} and \ref{thm:hybrid}.

As an immediate consequence of the Graham--Sloane lower bound for sparse paving matroids,
Theorem \ref{thm:main} gives the following explicit estimate.

\begin{corollary}\label{cor:GS-scale}
We have
\[
        p_n-sp_n
        \ge
        2^{\frac{1-o(1)}{n}\binom n{\lfloor n/2\rfloor}}.
\]
\end{corollary}

The final section explains why these results do not constitute a ratio-scale attack on the
Mayhew--Newman--Welsh--Whittle conjecture. More precisely, we prove that a
one-large-hyperplane method would need a ratio estimate of the form
\[
        \frac{i(J(n,r)[V_H])}{i(J(n,r))}
        =
        \omega\left(\binom n{r+1}^{-1}\right),
        \qquad
        \bigl(r=\lfloor n/2\rfloor,\ |H|=r+1\bigr),
\]
where \(i(G)\) denotes the number of stable sets of a graph \(G\). Since
\[
        \binom n{r+1}=2^{n-\frac12\log_2 n+O(1)}
\]
in the middle rank, this requires a lower bound at essentially reciprocal-binomial scale.
The currently available entropy and Graham--Sloane estimates preserve the logarithmic
scale but fall far short of this ratio-scale requirement.

All logarithms may be taken in any fixed base. For definiteness, one may take them to be
base \(2\). Since \(p_n-sp_n\) counts non-sparse paving matroids over all ranks, every
fixed-rank construction below gives a valid lower bound for \(p_n-sp_n\).

\section{Preliminaries}

Let \(E\) be a finite set and let \(0<r<|E|\). The \emph{Johnson graph} \(J(E,r)\) is the
graph whose vertex set is
\[
        \binom Er=\{X\subseteq E: |X|=r\},
\]
with two vertices \(X,Y\) adjacent if and only if \(|X\cap Y|=r-1\). When \(E=[n]\), we
write \(J(n,r)\) for \(J([n],r)\). For a graph \(G\), let \(i(G)\) denote the number of
stable sets of \(G\), including the empty stable set.

We shall use the standard characterization of sparse paving matroids by stable sets in
Johnson graphs.

\begin{lemma}\label{lem:sparse-paving-stable}
Let \(E\) be a finite set and let \(0<r<|E|\). A subset \(\cS\subseteq\binom Er\) is a
stable set in \(J(E,r)\) if and only if
\[
        \binom Er\setminus \cS
\]
is the set of bases of a rank-\(r\) sparse paving matroid on \(E\). Consequently,
\[
        sp_{n,r}=i(J(n,r)).
\]
\end{lemma}

\begin{proof}
This is the standard sparse-paving/Johnson-stable-set correspondence; see
Piff and Welsh \cite{PW70} and Pendavingh--van der Pol \cite{PvdP15,PvdP16}. We recall
the argument. In a rank-\(r\) sparse paving matroid, the rank-\(r\) nonbases are
circuit-hyperplanes. Two distinct circuit-hyperplanes cannot meet in \(r-1\) elements,
because then the circuit-elimination axiom would give a circuit of size at most \(r-1\),
contradicting the paving property. Hence the rank-\(r\) nonbases form a stable set in
\(J(E,r)\).

Conversely, if \(\cS\subseteq\binom Er\) is stable, then the family
\(\binom Er\setminus\cS\) satisfies the basis exchange axiom and defines a sparse paving
matroid whose circuit-hyperplanes are precisely the members of \(\cS\). Indeed, if
\(B_1,B_2\in\binom Er\setminus\cS\), \(e\in B_1\setminus B_2\), and
\(B_1-e+f\in\cS\) for every \(f\in B_2\setminus B_1\), then any two of these sets meet in
at least \(r-1\) elements after varying one element of the fixed set \(B_1-e\), forcing two
members of \(\cS\) to be adjacent unless \(|B_2\setminus B_1|=1\). In the remaining case
\(B_2=B_1-e+f\), this set is a basis by assumption. Thus basis exchange holds. The paving
and sparse-paving properties are the usual consequences of the fact that the only
rank-\(r\) nonbases are pairwise nonadjacent \(r\)-sets. This is the classical
construction.
\end{proof}

We shall also use the following lemma of Pendavingh and van der Pol. It follows from
Shearer's entropy lemma and applies to induced subgraphs of vertex-transitive graphs.

\begin{lemma}\label{lem:PvdP-induced}
Let \(G=(V,E)\) be a finite vertex-transitive graph, and let \(\varnothing\ne U\subseteq V\).
Then
\[
        \frac{\log i(G)}{|V|}
        \le
        \frac{\log i(G[U])}{|U|}.
\]
Equivalently,
\[
        i(G[U])\ge i(G)^{|U|/|V|}.
\]
\end{lemma}

\begin{proof}
This is Lemma 4.2 of Pendavingh and van der Pol \cite{PvdP16}. For completeness, we recall
the standard proof. Let \(\Gamma\) be the automorphism group of \(G\). For \(g\in\Gamma\),
put \(U^g=\{g(u):u\in U\}\). Since \(G\) is vertex-transitive, each vertex of \(G\) lies in
the same number of sets \(U^g\). Thus the family \(\{U^g:g\in\Gamma\}\) is a regular cover
of \(V\). Applying Shearer's product theorem to stable sets of \(G\), and then restricting
stable sets to the subsets \(U^g\), gives
\[
        i(G)^k\le \prod_{g\in\Gamma} i(G[U^g]),
        \qquad
        k=|\Gamma|\frac{|U|}{|V|}.
\]
Since \(G[U^g]\cong G[U]\) for every \(g\), this becomes
\[
        i(G)^{|\Gamma||U|/|V|}\le i(G[U])^{|\Gamma|}.
\]
Taking logarithms gives the result.
\end{proof}

The next lemma is the basic matroid construction used throughout the paper.

\begin{lemma}\label{lem:hyperplane-construction}
Let \(E\) be a finite set and let \(2\le r<|E|\). Let \(\cH\) be a family of proper subsets
of \(E\) such that
\[
        r\le |H|\le |E|-1
        \qquad(H\in\cH),
\]
and
\[
        |H_1\cap H_2|\le r-2
        \qquad(H_1\ne H_2,\ H_i\in\cH).
\]
Assume that at least one \(r\)-subset of \(E\) is not contained in any member of \(\cH\).
Define
\[
        \cB=
        \left\{B\in\binom Er:
        B\not\subseteq H \text{ for every }H\in\cH
        \right\}.
\]
Then \(\cB\) is the set of bases of a rank-\(r\) paving matroid \(M\) on \(E\). Moreover,
the hyperplanes of \(M\) having size at least \(r\) are exactly the members of \(\cH\).
\end{lemma}

\begin{proof}
By assumption, \(\cB\) is nonempty. We first verify the basis exchange axiom. Let
\(B_1,B_2\in\cB\), and let \(e\in B_1\setminus B_2\). Suppose, for contradiction, that
\[
        B_1-e+f\notin\cB
        \qquad
        \text{for every } f\in B_2\setminus B_1 .
\]
Then for each \(f\in B_2\setminus B_1\), there is \(H_f\in\cH\) such that
\[
        B_1-e+f\subseteq H_f.
\]
All the \(r\)-sets \(B_1-e+f\) contain the common \((r-1)\)-set \(B_1-e\). Since two
distinct members of \(\cH\) cannot contain the same \((r-1)\)-subset, all \(H_f\) are equal
to a single member \(H_0\in\cH\). Hence \(H_0\) contains \(B_1-e\) and every element of
\(B_2\setminus B_1\). Since
\[
        B_2=(B_1\cap B_2)\cup(B_2\setminus B_1)
        \subseteq (B_1-e)\cup(B_2\setminus B_1),
\]
we get \(B_2\subseteq H_0\), contradicting \(B_2\in\cB\). Therefore some
\(f\in B_2\setminus B_1\) satisfies \(B_1-e+f\in\cB\). Thus \(\cB\) is the set of bases of
a rank-\(r\) matroid \(M\) on \(E\).

We next show that \(M\) is paving. Let \(Y\in\binom E{r-1}\). If \(Y\) is not contained in
any member of \(\cH\), then for any \(x\in E\setminus Y\), the set \(Y+x\) is not contained
in any member of \(\cH\), and hence \(Y+x\in\cB\). If \(Y\subseteq H\) for some
\(H\in\cH\), then such an \(H\) is unique. Since \(H\) is proper, choose
\(x\in E\setminus H\). Then \(Y+x\) is not contained in \(H\), and it is not contained in
any other member of \(\cH\), since otherwise that other member would contain \(Y\). Hence
\(Y+x\in\cB\). Thus every \((r-1)\)-subset of \(E\) is contained in a basis. It follows
that every set of size at most \(r-1\) is independent, because any such set extends to an
\((r-1)\)-subset of \(E\), and that \((r-1)\)-subset is contained in a basis. Therefore
every circuit has size at least \(r\), so \(M\) is paving.

Now let \(H\in\cH\). No \(r\)-subset of \(H\) is a basis, so \(r_M(H)\le r-1\). Since every
\((r-1)\)-subset is independent and \(|H|\ge r\), we have \(r_M(H)=r-1\). If
\(x\in E\setminus H\), choose any \((r-1)\)-subset \(Y\subseteq H\). By the preceding
paragraph, \(Y+x\) is a basis. Hence \(r_M(H\cup\{x\})=r\). Therefore \(H\) is a
hyperplane of \(M\).

Conversely, let \(F\) be a hyperplane of \(M\) with \(|F|\ge r\). Since \(F\) has rank
\(r-1\), no \(r\)-subset of \(F\) is a basis. Hence every \(r\)-subset of \(F\) is contained
in some member of \(\cH\). The Johnson graph on \(\binom Fr\) is connected whenever
\(|F|\ge r\). Thus any two \(r\)-subsets of \(F\) can be joined by a sequence of
\(r\)-subsets of \(F\) in which consecutive sets intersect in exactly \(r-1\) elements.
Consecutive sets in such a sequence must be contained in the same member of \(\cH\), since
two distinct members of \(\cH\) cannot share an \((r-1)\)-subset. Therefore all \(r\)-subsets
of \(F\) are contained in one fixed member \(H_0\in\cH\). Since every element of \(F\) lies
in some \(r\)-subset of \(F\), we have \(F\subseteq H_0\). But \(H_0\) is a hyperplane by
the preceding paragraph, and \(F\) is also a hyperplane. Since hyperplanes are maximal
proper flats of rank \(r-1\), the inclusion \(F\subseteq H_0\) implies \(F=H_0\).
\end{proof}

\section{One large hyperplane}

Fix integers \(n,r,t\) with
\[
        2\le r<n,
        \qquad
        1\le t\le n-r-1.
\]
Let \(E=[n]\), and fix \(H\subseteq E\) with \(|H|=r+t\). Define
\[
        V_H=
        \left\{X\in\binom Er: |X\cap H|\le r-2\right\}.
\]
Thus \(V_H\) consists of those \(r\)-sets that do not share an \((r-1)\)-subset with \(H\).

Let \(I\) be a stable set in \(J(E,r)[V_H]\). Then \(\{H\}\cup I\) satisfies the
pairwise-intersection condition in Lemma \ref{lem:hyperplane-construction}. We shall also
need the following elementary nonemptiness observation.

\begin{lemma}\label{lem:nonempty-one-H}
Let \(E=[n]\), let \(2\le r<n\), and let \(H\subseteq E\) have size \(r+t\), where
\(1\le t\le n-r-1\). If \(I\subseteq V_H\), then at least one \(r\)-subset of \(E\) is not
contained in any member of \(\{H\}\cup I\).
\end{lemma}

\begin{proof}
Choose \(Y\subseteq H\) with \(|Y|=r-1\), and choose \(x\in E\setminus H\). This is possible
because \(H\ne E\). Put
\[
        B=Y\cup\{x\}.
\]
Then \(B\not\subseteq H\). Also \(B\notin V_H\), since \(|B\cap H|=r-1\). If
\(I\subseteq V_H\), then no member of \(I\) is equal to \(B\). Since every member of \(I\)
has size \(r\), the assertion \(B\subseteq X\) for \(X\in I\) would force \(B=X\), which is
impossible. Hence \(B\) is not contained in any member of \(\{H\}\cup I\).
\end{proof}

Thus Lemma \ref{lem:hyperplane-construction} gives a rank-\(r\) paving matroid from every
stable set \(I\) in \(J(E,r)[V_H]\). Since \(H\) is a hyperplane of size greater than \(r\),
this matroid is not sparse paving by the observation in the introduction.

\begin{theorem}\label{thm:induced-count}
Let \(E=[n]\), let \(2\le r<n\), and let \(1\le t\le n-r-1\). Fix \(H\subseteq E\) with
\(|H|=r+t\). Let \(q_{n,r,t}(H)\) denote the number of rank-\(r\) paving matroids \(M\) on
\(E\) such that \(H\) is a hyperplane of \(M\) and \(M\) is not sparse paving. Then
\[
        q_{n,r,t}(H)\ge i(J(n,r)[V_H]).
\]
If \(V_H\ne\varnothing\), then
\[
        q_{n,r,t}(H)\ge sp_{n,r}^{|V_H|/\binom nr}.
\]
\end{theorem}

\begin{proof}
Every stable set \(I\) of \(J(n,r)[V_H]\) gives, by Lemmas
\ref{lem:nonempty-one-H} and \ref{lem:hyperplane-construction}, a rank-\(r\) paving
matroid \(M_I\) in which \(H\) is a hyperplane. Since \(|H|=r+t>r\), each such matroid is
not sparse paving.

We show that distinct stable sets give distinct matroids. By Lemma
\ref{lem:hyperplane-construction}, the hyperplanes of \(M_I\) having size at least \(r\) are
exactly the members of \(\{H\}\cup I\). Since \(H\) has size \(r+t>r\) and every member of
\(I\) has size exactly \(r\), the set \(I\) is recoverable from \(M_I\) as the collection of
hyperplanes of size exactly \(r\). Thus the map \(I\mapsto M_I\) is injective, and
\[
        q_{n,r,t}(H)\ge i(J(n,r)[V_H]).
\]

If \(V_H\ne\varnothing\), then we may apply Lemma \ref{lem:PvdP-induced} to the
vertex-transitive graph \(G=J(n,r)\) and the nonempty subset \(U=V_H\). Since
\(i(J(n,r))=sp_{n,r}\) by Lemma \ref{lem:sparse-paving-stable}, we obtain
\[
        i(J(n,r)[V_H])
        \ge
        i(J(n,r))^{|V_H|/\binom nr}
        =
        sp_{n,r}^{|V_H|/\binom nr}.
\]
\end{proof}

\begin{lemma}\label{lem:VH-size}
With \(E=[n]\), \(2\le r<n\), \(1\le t\le n-r-1\), and \(|H|=r+t\), we have
\[
        |V_H|=
        \binom nr
        -\binom{r+t}{r}
        -\binom{r+t}{r-1}(n-r-t).
\]
\end{lemma}

\begin{proof}
The vertices excluded from \(V_H\) are exactly the \(r\)-subsets \(X\subseteq E\) satisfying
\(|X\cap H|\ge r-1\). Since \(|X|=r\), either \(X\subseteq H\), giving
\(\binom{r+t}{r}\) choices, or \(|X\cap H|=r-1\) and \(X\) contains one element outside
\(H\), giving \(\binom{r+t}{r-1}(n-r-t)\) choices. These two cases are disjoint and
exhaustive.
\end{proof}

\begin{corollary}\label{cor:explicit-r-t}
Let \(E=[n]\), let \(2\le r<n\), and let \(1\le t\le n-r-1\). Fix \(H\subseteq E\) with
\(|H|=r+t\). If \(V_H\ne\varnothing\), then
\[
        q_{n,r,t}(H)
        \ge
        sp_{n,r}^{
        1-
        \frac{\binom{r+t}{r}+\binom{r+t}{r-1}(n-r-t)}{\binom nr}}
        .
\]
\end{corollary}

\begin{proof}
This follows immediately from Theorem \ref{thm:induced-count} and Lemma
\ref{lem:VH-size}.
\end{proof}

\begin{proof}[Proof of Theorem \ref{thm:main}]
Take \(r=\lfloor n/2\rfloor\) and \(t=1\). For all sufficiently large \(n\), we have
\(1\le t\le n-r-1\) and \(V_H\ne\varnothing\). Fix \(H\subseteq[n]\) with \(|H|=r+1\). By
Corollary \ref{cor:explicit-r-t},
\[
        p_n-sp_n
        \ge
        q_{n,r,1}(H)
        \ge
        sp_{n,r}^{
        1-
        \frac{(r+1)+\binom{r+1}{2}(n-r-1)}{\binom nr}}
        .
\]
The numerator is \(O(n^3)\), while
\[
        \binom nr=\binom n{\lfloor n/2\rfloor}\sim \frac{2^n}{\sqrt{\pi n/2}}.
\]
Therefore the displayed fraction is \(o(1)\), and
\[
        p_n-sp_n\ge sp_{n,\lfloor n/2\rfloor}^{1-o(1)}.
\]
Taking logarithms gives the equivalent logarithmic form.
\end{proof}

\begin{proof}[Proof of Theorem \ref{thm:amplified-one}]
Apply Corollary \ref{cor:explicit-r-t} with \(t=1\). For each fixed \((r+1)\)-subset
\(H\subseteq[n]\), the construction gives at least \(sp_{n,r}^{1-\delta_n}\) rank-\(r\)
paving matroids on \([n]\) that are not sparse paving and in which \(H\) is a hyperplane.

The families obtained from distinct choices of \(H\) are disjoint. Indeed, if \(M\) is
produced from a stable set \(I\subseteq V_H\), then Lemma
\ref{lem:hyperplane-construction} says that the hyperplanes of \(M\) having size at least
\(r\) are exactly \(\{H\}\cup I\). Here \(H\) has size \(r+1\), while every member of \(I\)
has size exactly \(r\). Hence \(H\) is uniquely recoverable from \(M\) as the unique
hyperplane of size \(r+1\). Therefore no matroid produced from \(H\) can also be produced
from a different \((r+1)\)-subset \(H'\).

There are \(\binom n{r+1}\) choices for \(H\). Hence
\[
        p_n-sp_n\ge \binom n{r+1}sp_{n,r}^{1-\delta_n}.
\]
Finally, for \(r=\lfloor n/2\rfloor\), \(\delta_n=o(1)\), by the same estimate used in the
proof of Theorem \ref{thm:main}.
\end{proof}

\begin{proof}[Proof of Corollary \ref{cor:GS-scale}]
Graham and Sloane \cite{GS80} proved that
\[
        sp_{n,r}\ge 2^{\binom nr/n}
\]
for every \(0<r<n\). Applying this with \(r=\lfloor n/2\rfloor\), Theorem
\ref{thm:main} gives
\[
        p_n-sp_n
        \ge
        \left(2^{\binom n{\lfloor n/2\rfloor}/n}\right)^{1-o(1)}
        =
        2^{\frac{1-o(1)}{n}\binom n{\lfloor n/2\rfloor}}.
\]
\end{proof}

\begin{remark}\label{rem:not-ratio}
Theorem \ref{thm:main} is a logarithmic-scale lower bound. It does not imply that
\(sp_n/p_n\to0\). A ratio-scale disproof of the Mayhew--Newman--Welsh--Whittle conjecture
by this method would require substantially sharper information about how many stable sets
of \(J(n,r)\) remain after deleting the local forbidden neighborhood of a large hyperplane.
\end{remark}

\section{Multiple large hyperplanes}

We now describe a second way to push the construction. Instead of prescribing one large
hyperplane, we prescribe a family of large hyperplanes. The price is that these large
hyperplanes must satisfy a stronger packing condition.

Let \(G_{n,r+1}^{(6)}\) be the graph whose vertices are the \((r+1)\)-subsets of \([n]\),
with two distinct vertices \(C_1,C_2\) adjacent if
\[
        |C_1\cap C_2|\ge r-1.
\]
Thus a stable set in \(G_{n,r+1}^{(6)}\) is precisely a family
\(\cC\subseteq\binom{[n]}{r+1}\) satisfying
\[
        |C_1\cap C_2|\le r-2
        \qquad(C_1\ne C_2).
\]
Equivalently, \(\cC\) is a constant-weight code of length \(n\), weight \(r+1\), and
minimum Hamming distance at least \(6\).

\begin{theorem}[Distance-six large-hyperplane construction]\label{thm:distance-six}
Let \(2\le r\le n-3\). Then
\[
        p_n-sp_n\ge i(G_{n,r+1}^{(6)})-1.
\]
In particular,
\[
        p_n-sp_n\ge 2^{\alpha(G_{n,r+1}^{(6)})}-1.
\]
Moreover,
\[
        p_n-sp_n
        \ge
        2^{
        \displaystyle
        \frac{\binom n{r+1}}
        {1+(r+1)(n-r-1)+\binom{r+1}{2}\binom{n-r-1}{2}}
        }
        -1.
\]
\end{theorem}

\begin{proof}
Let \(\varnothing\ne\cC\) be a stable set in \(G_{n,r+1}^{(6)}\). Then every two distinct
members of \(\cC\) meet in at most \(r-2\) elements. We first check the nonemptiness
hypothesis of Lemma \ref{lem:hyperplane-construction}. No \((r-1)\)-subset of \([n]\) can
be contained in two distinct members of \(\cC\), since that would force those two members to
intersect in at least \(r-1\) elements. Hence
\[
        |\cC|\binom{r+1}{r-1}\le \binom n{r-1}.
\]
Also, no \(r\)-subset is contained in two distinct members of \(\cC\). Thus the number of
\(r\)-subsets contained in members of \(\cC\) is exactly \((r+1)|\cC|\), and therefore at
most
\[
        (r+1)\frac{\binom n{r-1}}{\binom{r+1}{2}}
        =
        \frac{2}{r}\binom n{r-1}.
\]
Since
\[
        \binom nr=\frac{n-r+1}{r}\binom n{r-1}
\]
and \(r\le n-3\), we have \(n-r+1\ge4\). Therefore the number of covered \(r\)-subsets is
strictly smaller than \(\binom nr\). Hence at least one \(r\)-subset is not contained in any
member of \(\cC\).

By Lemma \ref{lem:hyperplane-construction}, the family \(\cC\) defines a rank-\(r\) paving
matroid whose hyperplanes of size at least \(r\) are exactly the members of \(\cC\). Since
\(\cC\ne\varnothing\) and every member of \(\cC\) has size \(r+1\), the matroid is not
sparse paving. The map \(\cC\mapsto M(\cC)\) is injective, because \(\cC\) is recoverable
as the set of hyperplanes of size \(r+1\). This gives
\[
        p_n-sp_n\ge i(G_{n,r+1}^{(6)})-1.
\]
Since every subset of a stable set is stable,
\[
        i(G_{n,r+1}^{(6)})\ge 2^{\alpha(G_{n,r+1}^{(6)})}.
\]

It remains to estimate the independence number. The graph \(G_{n,r+1}^{(6)}\) has
\(\binom n{r+1}\) vertices. Fix a vertex \(C\). A distinct vertex \(D\) is adjacent to
\(C\) if either \(|C\cap D|=r\) or \(|C\cap D|=r-1\). The number of such \(D\) with
\(|C\cap D|=r\) is \((r+1)(n-r-1)\), and the number with \(|C\cap D|=r-1\) is
\(\binom{r+1}{2}\binom{n-r-1}{2}\). Thus
\[
        \Delta(G_{n,r+1}^{(6)})
        =
        (r+1)(n-r-1)+\binom{r+1}{2}\binom{n-r-1}{2}.
\]
The greedy bound \(\alpha(G)\ge |V(G)|/(\Delta(G)+1)\) gives the claimed explicit lower
bound.
\end{proof}

\begin{corollary}\label{cor:distance-six-middle}
For \(r=\lfloor n/2\rfloor\),
\[
        p_n-sp_n
        \ge
        2^{\Omega\left(\binom n{\lfloor n/2\rfloor}/n^4\right)}.
\]
\end{corollary}

\begin{proof}
For \(r=\lfloor n/2\rfloor\), the denominator in Theorem \ref{thm:distance-six} is
\(O(n^4)\), while \(\binom n{r+1}\) is within a constant factor of
\(\binom n{\lfloor n/2\rfloor}\). The result follows.
\end{proof}

For a stable set \(\cC\) in \(G_{n,r+1}^{(6)}\), define
\[
        V_{\cC}=
        \left\{
        X\in\binom{[n]}r:
        |X\cap C|\le r-2 \text{ for every } C\in\cC
        \right\}.
\]

\begin{theorem}[Hybrid large-hyperplane/stable-set construction]\label{thm:hybrid}
Let \(2\le r\le n-3\). Then
\[
        p_n-sp_n
        \ge
        \sum_{\varnothing\ne\cC}
        i\bigl(J(n,r)[V_{\cC}]\bigr),
\]
where the sum is over all nonempty stable sets \(\cC\) of \(G_{n,r+1}^{(6)}\). Consequently,
\[
        p_n-sp_n
        \ge
        \sum_{\substack{\varnothing\ne\cC\\ V_{\cC}\ne\varnothing}}
        sp_{n,r}^{|V_{\cC}|/\binom nr},
\]
again with \(\cC\) ranging over stable sets of \(G_{n,r+1}^{(6)}\).
\end{theorem}

\begin{proof}
Fix a nonempty stable set \(\cC\) of \(G_{n,r+1}^{(6)}\), and let \(I\) be a stable set in
\(J(n,r)[V_{\cC}]\). The family \(\cC\cup I\) satisfies the pairwise-intersection condition
in Lemma \ref{lem:hyperplane-construction}: two members of \(\cC\) meet in at most \(r-2\)
elements; two members of \(I\) meet in at most \(r-2\) elements; and if \(C\in\cC\) and
\(X\in I\), then \(|C\cap X|\le r-2\) by the definition of \(V_{\cC}\).

We check the nonemptiness hypothesis. Choose \(C\in\cC\), choose \(Y\subseteq C\) with
\(|Y|=r-1\), and choose \(x\in[n]\setminus C\). Such an \(x\) exists because
\(|C|=r+1\le n-2\). Put \(B=Y\cup\{x\}\). Then \(B\not\subseteq C\). If \(B\subseteq C'\)
for some \(C'\in\cC\) with \(C'\ne C\), then \(|C\cap C'|\ge |Y|=r-1\), contradicting the
stability of \(\cC\). Hence \(B\) is not contained in any member of \(\cC\). Also,
\(|B\cap C|=r-1\), so \(B\notin V_{\cC}\). Since \(I\subseteq V_{\cC}\) and every member of
\(I\) has size \(r\), the set \(B\) is not contained in any member of \(I\). Therefore
\(B\) is not contained in any member of \(\cC\cup I\).

Lemma \ref{lem:hyperplane-construction} now gives a rank-\(r\) paving matroid
\(M(\cC,I)\). It is not sparse paving because \(\cC\ne\varnothing\), so the matroid has a
hyperplane of size \(r+1\). The map \((\cC,I)\mapsto M(\cC,I)\) is injective: by Lemma
\ref{lem:hyperplane-construction}, the hyperplanes of size at least \(r\) are exactly
\(\cC\cup I\), and \(\cC\) is the subfamily of size \(r+1\), while \(I\) is the subfamily of
size \(r\). Hence both \(\cC\) and \(I\) are recoverable from the matroid. Summing over all
nonempty \(\cC\) gives the first inequality.

If \(V_{\cC}\ne\varnothing\), then Lemma \ref{lem:PvdP-induced} gives
\[
        i(J(n,r)[V_{\cC}])
        \ge
        i(J(n,r))^{|V_{\cC}|/\binom nr}
        =
        sp_{n,r}^{|V_{\cC}|/\binom nr}.
\]
Summing these inequalities proves the second assertion.
\end{proof}

\section{Ratio-scale limitations of the construction}

The preceding results give large lower bounds, but they do not by themselves disprove the
Mayhew--Newman--Welsh--Whittle conjecture. This section records the precise ratio-scale
obstruction for the one-large-hyperplane method.

Let \(r=\lfloor n/2\rfloor\), let \(H\subseteq[n]\) have size \(r+1\), and put
\[
        F_H=\binom{[n]}r\setminus V_H
        =
        \left\{X\in\binom{[n]}r: |X\cap H|\ge r-1\right\}.
\]
By Lemma \ref{lem:VH-size},
\[
        |F_H|
        =
        (r+1)+\binom{r+1}{2}(n-r-1)
        =
        \Theta(n^3).
\]
Thus the induced subgraph \(J(n,r)[V_H]\) is obtained from \(J(n,r)\) by deleting only
polynomially many vertices. The difficulty is that stable-set counts are sensitive enough
that such a deletion may still cause a large loss at the ratio scale needed here.

\begin{proposition}[A sufficient ratio-scale condition]\label{prop:sufficient-ratio}
Let \(r=\lfloor n/2\rfloor\), and let \(H\subseteq[n]\) have size \(r+1\). If
\[
        \binom n{r+1}\, i(J(n,r)[V_H])
        =
        \omega\bigl(i(J(n,r))\bigr),
\]
then the rank-\(r\) paving matroids produced by the one-large-hyperplane construction and
not sparse paving outnumber rank-\(r\) sparse paving matroids. Equivalently, since
\(i(J(n,r))=sp_{n,r}\), it is enough to prove
\[
        \frac{i(J(n,r)[V_H])}{i(J(n,r))}
        =
        \omega\left(\binom n{r+1}^{-1}\right).
\]
In particular, any lower bound of the form
\[
        \frac{i(J(n,r)[V_H])}{i(J(n,r))}
        \ge
        2^{-cn+o(n)}
\]
with a fixed constant \(c<1\) would be sufficient. More sharply, a lower bound of the form
\[
        \frac{i(J(n,r)[V_H])}{i(J(n,r))}
        \ge
        2^{-n+\omega(\log n)}
\]
would also be sufficient.
\end{proposition}

\begin{proof}
For each \((r+1)\)-subset \(H'\subseteq[n]\), Theorem \ref{thm:induced-count} constructs
at least \(i(J(n,r)[V_{H'}])\) rank-\(r\) paving matroids that are not sparse paving and
have \(H'\) as a hyperplane. Since all \((r+1)\)-subsets are equivalent under the symmetric
group on \([n]\), the value \(i(J(n,r)[V_{H'}])\) is independent of \(H'\). By the
recoverability argument in the proof of Theorem \ref{thm:amplified-one}, the families
obtained from different \(H'\) are disjoint. Hence the construction produces at least
\[
        \binom n{r+1} i(J(n,r)[V_H])
\]
rank-\(r\) paving matroids that are not sparse paving. The number of rank-\(r\) sparse
paving matroids is \(sp_{n,r}=i(J(n,r))\). This proves the first two assertions.

For \(r=\lfloor n/2\rfloor\), Stirling's formula gives
\[
        \binom n{r+1}=2^{n-\frac12\log_2 n+O(1)}.
\]
Hence
\[
        \binom n{r+1}^{-1}
        =
        2^{-n+\frac12\log_2 n+O(1)}.
\]
Therefore any lower bound of the form \(2^{-cn+o(n)}\) with \(c<1\) is larger than
\(\binom n{r+1}^{-1}\) by an exponential factor, and so gives the desired
\(\omega\)-condition. Also, any lower bound of the form \(2^{-n+\omega(\log n)}\) is
larger than \(2^{-n+\frac12\log_2 n+O(1)}\) by a factor tending to infinity, and hence is
also sufficient.
\end{proof}

The next proposition records exactly what is certified by applying the
Pendavingh--van der Pol induced-subgraph lemma to the deleted-vertex subgraph \(J(n,r)[V_H]\).
It should be read as a limitation of that particular certification, not as an upper bound
on the true number of matroids produced by the construction.

\begin{proposition}[What the induced-subgraph lemma certifies]\label{prop:entropy-certified}
Let \(r=\lfloor n/2\rfloor\), and let \(H\subseteq[n]\) have size \(r+1\). The lower bound
obtained from Lemma \ref{lem:PvdP-induced} and then amplified over all choices of \(H\) is
\[
        \binom n{r+1} sp_{n,r}^{1-\delta_n},
        \qquad
        \delta_n=
        \frac{(r+1)+\binom{r+1}{2}(n-r-1)}{\binom nr}.
\]
Using only the Graham--Sloane lower bound to estimate \(sp_{n,r}\), the multiplier of
\(sp_{n,r}\) certified by this argument is at most
\[
        2^{n+O(\log n)-\Theta(n^2)}.
\]
\end{proposition}

\begin{proof}
The first assertion is exactly Theorem \ref{thm:amplified-one}. Dividing the certified
lower bound by \(sp_{n,r}\) gives the certified multiplier
\[
        \binom n{r+1} sp_{n,r}^{-\delta_n}.
\]
For \(r=\lfloor n/2\rfloor\),
\[
        \binom n{r+1}=2^{n+O(\log n)}
\]
and
\[
        \delta_n\binom nr/n
        =
        \frac{(r+1)+\binom{r+1}{2}(n-r-1)}{n}
        =
        \Theta(n^2).
\]
Since the Graham--Sloane lower bound gives
\[
        sp_{n,r}\ge 2^{\binom nr/n},
\]
we have
\[
        sp_{n,r}^{-\delta_n}
        \le
        2^{-\delta_n\binom nr/n}
        =
        2^{-\Theta(n^2)}.
\]
Therefore, under this certification, the multiplier of \(sp_{n,r}\) is at most
\[
        2^{n+O(\log n)}2^{-\Theta(n^2)}
        =
        2^{n+O(\log n)-\Theta(n^2)}.
\]
\end{proof}

We also record the corresponding Graham--Sloane color-class version. This gives a fully
explicit subfamily of the construction, but it still remains logarithmic-scale rather than
ratio-scale.

\begin{proposition}[Graham--Sloane color classes inside \(V_H\)]\label{prop:GS-punctured}
Let \(0<r<n\), and define
\[
        f(X)=\sum_{x\in X}x \pmod n
        \qquad
        \left(X\in\binom{[n]}r\right).
\]
For \(k\in\Z/n\Z\), put
\[
        U_{n,r,k}=\left\{X\in\binom{[n]}r:f(X)=k\right\}.
\]
Then each \(U_{n,r,k}\) is a stable set of \(J(n,r)\). If \(H\subseteq[n]\) has size
\(r+1\), then for some \(k\in\Z/n\Z\),
\[
        |U_{n,r,k}\cap V_H|
        \ge
        \frac{1}{n}
        \left(
        \binom nr-(r+1)-\binom{r+1}{2}(n-r-1)
        \right).
\]
Consequently, for \(r=\lfloor n/2\rfloor\),
\[
        p_n-sp_n
        \ge
        \binom n{r+1}
        2^{\binom nr/n-\Theta(n^2)}.
\]
\end{proposition}

\begin{proof}
If \(X\) and \(Y\) are adjacent in \(J(n,r)\), then \(Y=X-a+b\) for distinct
\(a,b\in[n]\), so
\[
        f(Y)-f(X)\equiv b-a\not\equiv 0 \pmod n.
\]
Thus each fiber \(U_{n,r,k}\) is stable. Since the \(n\) fibers partition
\(\binom{[n]}r\), the sets \(U_{n,r,k}\cap V_H\) partition \(V_H\). Hence some fiber has
size at least \(|V_H|/n\). Lemma \ref{lem:VH-size}, with \(t=1\), gives
\[
        |V_H|=\binom nr-(r+1)-\binom{r+1}{2}(n-r-1).
\]
Every subset of the stable set \(U_{n,r,k}\cap V_H\) is again stable in \(J(n,r)[V_H]\), so
the one-large-hyperplane construction gives at least
\[
        2^{|U_{n,r,k}\cap V_H|}
\]
matroids for this fixed \(H\). Amplifying over all choices of \(H\), as in Theorem
\ref{thm:amplified-one}, gives
\[
        p_n-sp_n
        \ge
        \binom n{r+1}
        2^{|V_H|/n}.
\]
For \(r=\lfloor n/2\rfloor\), the deleted term divided by \(n\) is
\[
        \frac{(r+1)+\binom{r+1}{2}(n-r-1)}{n}
        =
        \Theta(n^2),
\]
which gives the displayed bound.
\end{proof}

\begin{remark}\label{rem:barrier}
The amplified bounds above do not by themselves disprove the Mayhew--Newman--Welsh--Whittle
conjecture. For one large hyperplane of size \(r+1\), the forbidden set in \(J(n,r)\) has
size
\[
        (r+1)+\binom{r+1}{2}(n-r-1)=\Theta(n^3)
\]
when \(r=\lfloor n/2\rfloor\). The induced-subgraph lemma is strong enough to preserve the
logarithmic scale but not strong enough to prove a ratio-scale estimate. The Graham--Sloane
color-class calculation gives an explicit subfamily but loses a factor \(2^{\Theta(n^2)}\)
compared with the natural Graham--Sloane scale. Thus a ratio-scale attack through this
construction would require a much sharper estimate for \(i(J(n,r)[V_H])\), or for the
hybrid sum in Theorem \ref{thm:hybrid}.
\end{remark}

\begin{problem}\label{prob:ratio-local}
For \(r=\lfloor n/2\rfloor\) and \(|H|=r+1\), estimate the ratio
\[
        \frac{i(J(n,r)[V_H])}{i(J(n,r))}
\]
more sharply. In particular, determine whether this ratio can be bounded below at
reciprocal-binomial scale, for example by \(2^{-n+\omega(\log n)}\), or whether a loss of
order \(2^{-\Omega(n^2)}\) is unavoidable.
\end{problem}

\begin{problem}\label{prob:hybrid-partition}
For \(r=\lfloor n/2\rfloor\), estimate the hybrid sum
\[
        \sum_{\varnothing\ne\cC}
        i\bigl(J(n,r)[V_{\cC}]\bigr),
\]
where \(\cC\) ranges over stable sets of \(G_{n,r+1}^{(6)}\). Can this sum be compared at
ratio scale with \(sp_{n,r}\), or only at logarithmic scale?
\end{problem}

\section*{Acknowledgments}

I thank Prof. Rudi Pendavingh for suggesting the use of Lemma 4.2 of \cite{PvdP16} in the
induced Johnson subgraph considered here, and for comments motivating the
large-hyperplane amplification.

\end{document}